\documentclass[12pt]{amsart}
\usepackage{amsmath}
\usepackage{amssymb}

\theoremstyle{plain}

\newtheorem*{lemma}{Lemma}
\newtheorem*{proposition}{Proposition}
\newtheorem*{fact}{Fact}

\theoremstyle{definition}

\title[]{Addendum to: \\ ``An expansion of the Jones representation of genus 2 and the Torelli group''}

\author[Y. Kasahara]{Yasushi Kasahara}
\date{April 19, 2010}
\address{
Department of Mathematics \\
  Kochi University of Technology \\ Tosayamada,  Kami City, Kochi \\ 
  782-8502 Japan 
}

\email{kasahara.yasushi@kochi-tech.ac.jp}

\begin{document}
%
%
\newcommand{\Q}{\mathbb Q}
\newcommand{\C}{\mathbb C}
\newcommand{\D}{\mathcal{D}}
\newcommand{\g}{\gamma}
\newcommand{\G}{\Gamma}
\newcommand{\gr}{\mathcal{G}}
\newcommand{\Gd}{\gr \delta}
\newcommand{\HE}{\mathcal H}
\newcommand{\K}{\mathcal K}
\newcommand{\h}{\mathfrak h}
\newcommand{\I}{{\mathcal I}}
\renewcommand{\S}{\Sigma}
\newcommand{\sgn}{\operatorname{sgn}}
\newcommand{\rank}{\operatorname{rank}}
\newcommand{\Sp}{\operatorname{Sp}}
\newcommand{\PSp}{\operatorname{PSp}}
\newcommand{\GL}{\operatorname{GL}}
\newcommand{\ctr}{\operatorname{Center}}
\newcommand{\Int}{\operatorname{Int}}
\newcommand{\Homeo}{\operatorname{Homeo}}
\renewcommand{\L}{\mathcal L}
\newcommand{\M}{{\mathcal M}}
\newcommand{\N}{{\mathcal N}}
\newcommand{\p}{\psi}
\newcommand{\x}{\xi}
\newcommand{\z}{\zeta}
\newcommand{\Z}{{\mathbb Z}}
\newcommand{\Aut}{\operatorname{Aut}}
\newcommand{\End}{\operatorname{End}}
\newcommand{\tr}[1]{\operatorname{trace}{#1}}
\renewcommand{\hom}{\operatorname{Hom}}
\renewcommand{\ker}{\operatorname{Ker}}
\newcommand{\im}{\operatorname{Im}}
\newcommand{\var}[1]{\varphi^{(#1)}}
\newcommand{\IW}{{H(q, n)}}
\newcommand{\PSL}{\operatorname{PSL}}
\newcommand{\F}{\mathcal{F}}
\newcommand{\ppar}{\par\goodbreak\medskip} 
%
%
\bibliographystyle{amsplain}
\begin{abstract}
 We observe that the determinant of the representation provides a little 
 restriction for the structure of the graded quotients $\gr_k^\Q{\F}$ introduced 
in both \cite{AGT} and \cite{JKTR} 
 that any one of them does not contain the trivial $1$-dimensional summand.
\end{abstract}

\maketitle

In our previous papers \cite{AGT} and \cite{JKTR}, we studied the Jones representation of genus $2$ 
$$\rho: \M_2 \to \GL(5, \Z[t,t^{-1}])$$
where $\M_2$ denotes the mapping class group of an closed oriented surface of genus $2$. 
We considered certain expansions of the involved parameter starting with the substitutions 
$t = -e^h$ in the former paper \cite{AGT}, and $t = e^h$ in the latter one  \cite{JKTR}, 
which induce two {\em different} descending central filtrations, both denoted by the same symbol
$\{ \F_k \}_{k \geq 1}$,  of the Torelli subgroup $\I_2 \subset \M_2$.
In each case, the associated graded quotients $\gr_k\F = \F_k/\F_{k+1}$
are free abelian groups of finite rank, and their scalar extensions $\gr^\Q_k\F = \gr_k\F \otimes \Q$
have natural structures of modules over $\Sp(4,\Q)$ and the symmetric group $\mathfrak{S}_6$ of 
degree $6$, respectively, and the direct sum $\bigoplus_{k \geq 1}{\gr^\Q_k\F}$ forms a graded Lie 
algebra. We determined,  in \cite{AGT} and \cite{JKTR}, respectively,
the Lie subalgebra generated by the homogeneous component of degree $1$. 
However, the structure of the graded quotients beyond this Lie subalgebra has remained unknown for 
the both cases. 
\par

In this note, we notify that the determinant of the representation $\rho$ provides a 
little restriction for the structure of each graded quotient for the both cases 
that it does not contain the trivial $1$-dimensional 
summand, which we had overlooked previously.
While our argument  is quite elementary, it is similar to that for the 
Morita trace for the images of the Johnson homomorphisms \cite{morita}. 
\ppar
From now on, we freely use the notations in our previous papers \cite{AGT} and \cite{JKTR}, which are 
common, as well as our argument below, except for the initial substitutions $t = -e^h$ and $t = e^h$, 
respectively. 
We begin our observation with:
\begin{fact}
 The determinant of the Jones representation $\rho$ lies in $\{ \pm 1 \}$.
\end{fact}

This fact follows from the construction by Jones \cite{jones}, and in fact 
holds for all the representations of the mapping class groups of punctured spheres 
constructed in \cite{jones} from rectangular Young diagrams. For the genus $2$ case, however,  one 
can verify the fact easily by using the explicit matrix form given in \cite{jones} or \cite{AGT}.
\par

It is well-known that the abelianization of $\M_2$ is $\Z/10\Z$, and one can easily see that 
the image of the Torelli subgroup $\I_2$ in the abelianization is $\Z/5\Z$. Therefore, the above fact 
immediately implies that the determinant of $\rho$ is identically $1$ on $\I_2$. Hence, for each $k \geq 1$,
we have
$$\det{\var{k}(\F_k)} \equiv 1 \mod{h^{k+1}}$$
On the other hand, for any $x \in \F_k$, $\var{k}(x)$ has the form
$$\var{k}(x) = I + h^k \cdot \Delta_k(x) \mod{R[k+1]}$$
where $I$ denotes the identity matrix, and $\Delta_k$ is a homomorphism of $\F_k$, with kernel $\F_{k+1}$,
into the additive group $M(5, \Q)$ of the matrix algebra of degree $5$ over $\Q$.
Considering $M(5,\Q)$ as an $\M_2$-module by $\varphi^{(0)} \otimes (\varphi^{(0)})^*$, $\Delta_k$ is 
also $\M_2$-equivariant.  We remark that the 
homomorphism $h^k \cdot \Delta_k$, with its target considered as $\gr_k{R} = R[k]/R[k+1]$, 
is denoted by $\delta_k$ in our previous papers. Essential for this note is the following:
\begin{lemma}
 For any $k \geq 1$ and $x \in \F_k$, 
 $$\det{\var{k}(x)} = 1 + h^{k} \cdot \tr{\Delta_k}(x) \mod{h^{k+1}}$$
where $\tr{}$ denotes the usual matrix trace.
\end{lemma}
\begin{proof}
 Writing  $I + h^k \cdot \Delta_k(x) = (a_{i,j})$, consider the definition for determinant
 $$\det{\var{k}(x)} = \sum_{\sigma \in \mathfrak{S}_5}{\sgn{\sigma}\prod_{i=1}^5{a_{i,\sigma(i)}}}.$$
 It is easy to see that  each summand except for $\sigma = 1$ has at least two factors from the 
 off diagonal,   and hence is a multiple of $h^{2k}$, which vanishes modulo ${h^{k+1}}$. 
 Therefore, $\det{\var{k}(x)}$ is simply the 
 product of the diagonal entries, which is equal to the right-hand side in the lemma.
\end{proof}
\par

Therefore, we see that $\Delta_k({\F_k})$ satisfies $\tr{} \equiv 0$.
For each of the cases \cite{AGT} and \cite{JKTR}, the $\M_2$-module $M(5,\Q)$, 
which is identical to $\gr_kR$, is in fact an $\Sp(4,\Q)$-module and an $\mathfrak{S}_6$-module,
respectively, and its irreducible decomposition is  explicitly given by 
\begin{align*}
 & \Gamma_{0,2} \oplus \Gamma_{2,0} \oplus \Gamma_{0,0}  \tag*{\cite[(4.1)]{AGT}}\\
 \intertext{and} 
 & [6] \oplus [4,2] \oplus [2^3] \oplus [3,1^3],  \tag*{\cite[(3.4)]{JKTR}}
\end{align*}
respectively. 

Among the summands in each case, the trivial $1$-dimensional one, which corresponds to 
$\Gamma_{0,0}$ and $[6]$, respectively, can be identified with the center of $M(5,\Q)$, 
and therefore has {\em nonzero} trace. This shows:
\begin{proposition}
 The trivial $1$-dimensional summand does not appear in $\gr_k^\Q{\F}$ for {\em any} $k \geq 1$.
\end{proposition}
\par

We remark that all the other summands of $M(5, \Q)$ are all trace free 
since  we already know by \cite[Theorem 1.3]{AGT} 
and \cite[Theorem 2.2]{JKTR}, respectively, that any non-trivial summand of $M(5,\Q)$ appears
in some homogeneous component of the Lie subalgebra generated by $\gr_1^\Q{\F}$.
\par

\subsection*{Acknowledgement.} The author is grateful to Professor Shigeyuki Morita for 
asking about the determinant of the Jones representation  for over ten years intermittently.

\bibliography{ref}
%
\end{document}